\documentclass[11pt]{amsart}

\usepackage{amsmath,amssymb,amscd,amsthm,amsxtra}
\usepackage{appendix}

\usepackage{graphicx, mathrsfs}
\usepackage{amssymb,amsmath,amsthm}
\usepackage{mathrsfs,dsfont}

\usepackage{amsmath}
\usepackage{amsthm}
\usepackage{latexsym}
\usepackage{amssymb,amsfonts}
\usepackage{xcolor}
\usepackage{marvosym}
\usepackage{bbm}
\usepackage{graphicx, color} 
\usepackage{palatino, url, multicol}

\usepackage[all]{xy}

\makeatletter
\DeclareRobustCommand\widecheck[1]{{\mathpalette\@widecheck{#1}}}
\def\@widecheck#1#2{%
   \setbox\z@\hbox{\m@th$#1#2$}%
   \setbox\tw@\hbox{\m@th$#1%
      \widehat{%
         \vrule\@width\z@\@height\ht\z@
         \vrule\@height\z@\@width\wd\z@}$}%
   \dp\tw@-\ht\z@
   \@tempdima\ht\z@ \advance\@tempdima2\ht\tw@ \divide\@tempdima\thr@@
   \setbox\tw@\hbox{%
      \raise\@tempdima\hbox{\scalebox{1}[-1]{\lower\@tempdima\box\tw@}}}%
   {\ooalign{\box\tw@ \cr \box\z@}}}
\makeatother

\allowdisplaybreaks[2]

\newtheorem{theorem}{Theorem} [section]

\newtheorem{lemma}[theorem]{Lemma}
\newtheorem{proposition}[theorem]{Proposition}
\newtheorem{remark}[theorem]{Remark}

\newtheorem{definition}[theorem]{Definition}

\newcommand{\Z}{\mathbb{Z}}

\newcommand{\R}{\mathbb{R}}

\newcommand{\T}{\mathbb{T}}
\newcommand{\PP}{\mathbb{P}}

\newcommand{\ff}{\vec f}
\newcommand{\vv}{\vec v}
\newcommand{\uu}{\vec u}
\newcommand{\ww}{\vec w}
\newcommand{\bk}{{\bf{k}}}
\newcommand{\be}{{\bf{e}}}

\newcommand{\DD}{\mathcal D}

\newcommand{\ft}{\widehat}

\numberwithin{equation}{section}

\begin{document}

\title[Supercritical Navier Stokes and a.s. global existence]{
Almost sure existence of global weak solutions for super-critical Navier-Stokes equations}

\author[Nahmod]{Andrea R. Nahmod$^1$}
\address{$^1$  
Department of Mathematics \\ University of Massachusetts\\  710 N. Pleasant Street, Amherst MA 01003}
\email{nahmod@math.umass.edu}
\thanks{$^1$ The first author is funded in part by NSF DMS 0803160.}

\author[Pavlovi\'c]{Nata\v{s}a Pavlovi\'c$^2$}
\address{$^2$  
Department of Mathematics\\ 
University of Texas at Austin\\ 
2515 Speedway, Stop C1200\\
Austin, TX 78712}
\email{natasa@math.utexas.edu}
\thanks{$^2$ The second author is funded in part by NSF DMS 0758247, NSF DMS 1101192 and an Alfred P. 
Sloan Research Fellowship.}

\author[Staffilani]{Gigliola Staffilani$^3$}
\address{$^3$ Department of Mathematics\\
Massachusetts Institute of Technology\\ 
77 Massachusetts Avenue,  Cambridge, MA 02139}
\email{gigliola@math.mit.edu}
\thanks{$^3$ The third author is funded in part by NSF
DMS 1068815.}

\date{}

\begin{abstract}
In this paper we show that after suitable data randomization there exists  a large set of super-critical periodic initial data, 
in $H^{-\alpha}({\mathbb T}^d)$ for some $\alpha(d) > 0$, for both 2d and 3d Navier-Stokes equations for which global energy bounds hold. 
As a consequence, we obtain almost sure large data super-critical global weak solutions. We also show that in 2d these global weak solutions are unique. 
\end{abstract}
\maketitle
\section{Introduction} 

Consider the initial value problem for the incompressible Navier-Stokes equations given by
\begin{equation}\label{NS}
\left\{\begin{array}{ll}
\partial_t \vec u =\Delta \vec u -\PP\nabla 
\cdot (\vec u\otimes\vec u); \qquad  x \in {\mathbb T}^{d} \, \, \mbox{or} \, \, \R^d, \, t>0\\
  \nabla \cdot \vec u=0 \\
  \vec u( x, 0)  = \vec f(x),
\end{array}\right.
\end{equation}
 where  $f $ is divergence free  and $\PP$ is the Leray projection into divergence free vector fields given via
\begin{equation}\label{div}
\PP\vec h=\vec h -\nabla\frac{1}{\Delta}(\nabla\cdot \vec h).
\end{equation}  

It is well-known  that global well-posedness of  \eqref{NS}
when the space dimension $d = 3$ 
is a long standing open question. 
This is related to the fact that the equations \eqref{NS} are so called 
super-critical when $d > 2$. Indeed, recall that   if the velocity vector field $\vec u(x,t)$ solves the Navier-Stokes
equations \eqref{NS} then $\vec u_{\lambda}(x,t)$ with
$$ \vec u_{\lambda}(x,t) = \lambda \vec u(\lambda x, \lambda^2 t),$$
is also a  solution to the system \eqref{NS},  for the initial data
\begin{equation} \label{scaluz} 
	\vec u_{0 \; \lambda} = \lambda \vec u_0 (\lambda x) \;\;. 
\end{equation}
The spaces which are
invariant under 
such a scaling are called critical spaces for the Navier-Stokes equations. Examples
of critical spaces for the Navier-Stokes are: 
\begin{equation} \label{embed} 
 	\dot{H}^{\frac{d}{2}-1} \hookrightarrow L^{d}  
	\hookrightarrow  \dot{B}_{p, \infty}^{-1+\frac{d}{p}}  
	\hookrightarrow BMO^{-1},  \quad \quad 1< p < \infty.
\end{equation} 

In particular, for Sobolev spaces,  
$ \| \vec u_{\lambda}(x, 0) \|_{\dot H^{s_c}} = \| \vec u(x, 0) \|_{\dot H^{s_c}}, $ 
when $s_c = \frac{d}{2} - 1$. We recall that the exponents  $s$ are called 
{\it critical} if $s = s_c$, {\it sub-critical} if $s > s_c$ and  {\it super-critical} if $s < s_c$. 

On the other hand, classical solutions to the \eqref{NS} satisfy the decay of energy 
which can be expressed as: 
\begin{equation} \label{energy} 
 \|u(x, t)\|_{L^2}^2 + \int_{0}^{t} \|\nabla u(x,\tau)\|_{L^2}^2 \; d\tau = \|u(x, 0)\|_{L^2}^2.
\end{equation} 
Note  that when $d=2$, the energy $ \|u(x, t)\|_{L^2}$, which is globally controlled thanks to \eqref{energy}, 
is  exactly  the scaling invariant $\dot H^{s_c}=L^2$-norm. In this case the equations are said to be  {\it critical}. 
When  $d=3$, the energy $ \|u(x, t)\|_{L^2}$ is at the super-critical level  with respect to the scaling invariant $\dot H^{\frac{1}{2}}$-norm, 
and hence the Navier-Stokes equations are said to be {\it super-critical} and the lack of a known bound for  the $\dot H^{\frac{1}{2}}$ contributes in keeping  the global well-posedness question for the initial value problem \eqref{NS} still open.

One way of studying the Navier-Stokes initial value problem \eqref{NS} is via weak solutions 
introduced in the context of  these equations by 
Leray \cite{Leray1, Leray2, Leray3} in 1930s. Leray \cite{Leray3} and Hopf \cite{Hopf} 
showed existence of global weak solutions of the Navier-Stokes equations corresponding to  initial data in 
$L^2({\mathbb R}^d)$.  When $d=2$ classical global solutions were later obtained  by Ladyzhenskaya \cite{Lady}.  
Lemari\'e-Rieusset generalized Leray's construction to prove  existence of uniformly locally square integrable weak solutions, for details see \cite{LR}. 
However questions addressing uniqueness and regularity of these solutions when $d=3$ have not 
been answered yet, although there are many important contributions in understanding partial regularity 
and conditional uniqueness of weak solutions, see e.g. Caffarelli-Kohn-Nirenberg \cite{CKN}, 
Lin \cite{Lin}, Escauriaza-Seregin-\v{S}verak \cite{ESS},  Vasseur \cite{Va}. 
Another  approach in studying existence of solutions to \eqref{NS} is to construct 
solutions to the corresponding integral equation via a fixed point theorem. In such a way one obtains 
so called `mild' solutions.  This approach was pioneered by Kato and Fujita,  see for example \cite{FUJK64}.  
However the existence of mild solutions to the Navier-Stokes equations in ${\mathbb R}^{d}$ for $d \geq 3$ 
has been obtained only locally in time and globally for small initial data in various sub-critical or critical spaces, see e.g. Kato
\cite{KAT84}, Cannone \cite{CANN-1, CANN-2}, Planchon \cite{PLAN},  Koch and Tataru \cite{KT01}, Gallagher and Planchon \cite{GP}, or globally in time under conditions on uniform in time boundedness of certain scaling invariant norms, see e.g. Kenig-Koch \cite{KeKo}.  In this context,  Cannone and Meyer \cite{CANN-MEY} proved that 
if  $f \in X$,  a {\it well-suited} Banach space for the study of the Navier-Stokes then the {\it fluctuation} $\ww: = \vec u - e^{t\Delta}\vec f$ is `better' in that it belongs to the Besov-type space ${\dot B}^0_{X, 1}$.

In this paper we consider the periodic Navier-Stokes problem in \eqref{NS} and in particular we address the question of long time existence  of weak solutions for super-critical initial data both in $d=2,\, 3$, see also Tao \cite{Tao}. For $d=2$ we address uniqueness as well. Our  goal  is to show that by randomizing in an appropriate way the  initial data in 
 $H^{-\alpha}({\mathbb T}^d), d=2, 3$ (for some $\alpha= \alpha(d)>0$) which is below the critical threshold 
space $H^{s_c}({\T}^d)$, as well as below the space $L^2$ where one has available deterministic constructions of weak solutions, one can construct  a global in time weak solution to \eqref{NS}. Such solution is unique when $d=2$.
Similar well-posededness results for randomized data were obtained for the super-critical nonlinear Schr\"odinger equation by Bourgain \cite{Bou2d} and for super-critical nonlinear wave equations by Burq and Tzvetkov in \cite{BTa, BTb, BTc}.
The approach of Burq and Tzvetkov was applied in the context of the Navier-Stokes in order to obtain local in time solutions to the corresponding integral 
equation for randomized initial data in $L^2({\mathbb T}^3)$, as well as global in time solutions to the corresponding integral equation for randomized initial data that are small 
in $L^2({\mathbb T}^3)$ by Zhang and Fang \cite{ZhFa-pams} and by Deng and Cui \cite{DC-1}.
Also in \cite{DC-2}, Deng and Cui obtained local in time solutions to the corresponding integral 
equation for randomized initial data in $H^s({\mathbb T}^d)$, for $d=2,3$ with $-1 < s < 0$.  

The paper at hand is the first to offer a construction of a {\bf global in time weak solution} to \eqref{NS} for randomized initial data (without any smallness assumption) in negative Sobolev spaces $H^{-\alpha}({\mathbb T}^d), d=2, 3$, for some $\alpha= \alpha(d) > 0$. 
Roughly speaking the idea of the proof is the following: we start with a divergence free and mean zero initial data $\vec f \in  (H^{- \alpha}({\T}^d))^d, d=2, 3$ 
and suitably randomize it to obtain $\vec f^{\omega}$ (see Definition \ref{defrand} for details) which in particular preserves the divergence free condition. 
Then we seek  a solution to the initial value problem \eqref{NS} in the form 
$\vec u = e^{t\Delta}\vec f^\omega + \ww$.  In this way,  the linear evolution $e^{t\Delta}\vec f^\omega$  is singled out and the difference equation that $\ww$ satisfies is identified.  At this point it becomes convenient to state the equivalence Lemma \ref{equivalence} between the initial value problem for the difference equation and the integral formulation of it.  This equivalence is similar to Theorem 11.2 in \cite{LR}, see also \cite{FLT}.  We will use the integral equation  formulation near time zero and the other one away from zero (see Section  \ref{sec-estimates}  for more details). The key point of this approach is the  fact  that although the initial data are in 
$H^{-\alpha}$ for some $\alpha>0$, the heat flow of the randomized data gives almost surely improved $L^p$ bounds (see Section \ref{sec-heatfl}).  These bounds in turn yield improved  nonlinear estimates arising in the analysis of the difference  equation for $\ww$  almost surely (see Section \ref{sec-estimates}  for details), 
and consequently a construction of a global weak solution to the difference equation is possible (see Section \ref{sec-galerkin} for details).

It is important to note that,  almost surely in $\omega$, the randomized initial data $\vec f^{\omega}$ belongs to $W^{- \alpha, p}$ for any $p\ge 1 $ and hence it is in $B^{-1 + \frac{d}{p}}_{p, q}$ for $ p$ large enough so that $1 - \frac{d}{p} > \alpha$,  and any $q \ge 2$. In particular, $\vec f^{\omega}$ belongs to the critical Besov spaces for which Gallagher and Planchon \cite{GP} proved, when $d=2$, global existence, uniqueness and suitable bounds.  Since $\vec f^{\omega}$ also belongs to $BMO^{-1}$,  small data (almost sure) well-posedness follows when $d=2, 3$ from Koch and Tataru's result \cite{KT01}.  
The goal of this paper however, is to show that there exists a large set of super-critical periodic initial data of arbitrary size in $H^{-\alpha}({\mathbb T}^d)$,  $d=2, 3$ that evolve to global solutions for which once the linear evolution  is removed we directly obtain energy bounds.

\subsection{Organization of the paper} In Section \ref{sec-notation} we introduce appropriate notation and state the main results. 
In Section \ref{sec-heatfl} we prove some useful bounds for  the heat flow on randomized data. In Section \ref{diff-equiv} we introduce the difference equation for $\ww$ and 
establish two equivalent formulations for the equation that $\ww$ solves. In  Section \ref{sec-estimates}  we prove energy estimates for $\ww$ and in Section \ref{sec-galerkin} we construct weak solutions 
to the difference equation via a Galerkin method. In Section \ref{sec-uniq} we prove uniqueness of weak solutions when $d=2$.  Finally in Section \ref{proof-main} we combine all the ingredients to establish the main theorems.
  
\subsection{Acknowledgements}
The authors would like to thank Cheng Yu for noticing that an earlier version
needed a revision in the energy estimates for the $d=3$ case.

\section{Notation and the statement of the main result} \label{sec-notation}

\subsection{Notation} \label{subsec-not}
In this subsection we list some notation that will be frequently used throughout the paper.
\begin{itemize}
\item Let $\mathcal{B}$ be a Banach  space of functions. The space $C_{\text{weak}}((0,T), \mathcal{B})$ denotes  the  subspace of $L^\infty((0,T), \mathcal{B})$ consisting of functions which are weakly continuous, i.e.  $v\in C_{\text{weak}}((0,T), \mathcal{B}) $ if and only if  $\phi(v(t))$ is a continuous function of $t$ for any $\phi\in \mathcal{B}^*$.
\item If $(X(\T^d))^d$ denotes a space of vector fields on $\T^d$, we simply denote its norm by $\|\cdot\|_X$.
\item We introduce an analogous notation to that of Constantin and Foias in \cite{CF}. In particular we write
\begin{align*} 
	H & = \mbox{the closure of } \{\vec f \in (C^{\infty}(\T^d))^d \, | \, \nabla \cdot \vec f =0 \} \mbox{ in } (L^2(\T^d))^d, \\ 
	V & = \mbox{the closure of } \{\ff \in (C^{\infty}(\T^d))^d \, | \, \nabla \cdot \ff =0 \} \mbox{ in }  ( H^1(\T^d))^d, \\
	V' &= \, \,\mbox{the dual of } V.
	\end{align*}
	\item We finally introduce some notation for the inner products in some of the spaces introduced above. The notation is similar to the one used in \cite{CF}.
	
	Given two vectors $\uu$ and $\vv$ in $\R^d$ we use the notation
	\begin{equation}\label{inner}
	\langle \uu,\vv\rangle=\uu\cdot\vv.
	\end{equation}
		In  $(L^2(\T^d))^d$ we use the inner product notation
	\begin{equation}\label{l2-inner}
	(\uu,\vv)=\int \uu(x)\cdot \vv(x)\,dx.
	\end{equation}
 In $(\dot H^1(\T^d))^d$ we use the inner product notation
\begin{equation}\label{h1-inner}
	((\uu,\vv))=\sum_{i=1}^d (D_i\uu,D_i\vv).
	\end{equation}
\item 	Finally we introduce the  trilinear expression 
\begin{equation}\label{tri}
b(\uu,\vv,\vec w)=\int \uu_jD_j\vv_i {\vec w}_i\, dx=\int \langle \uu \cdot \nabla \vv , \vec w \rangle\,dx.
\end{equation}
Also we note that when $\uu$ is divergence free, we have  
\begin{equation}\label{tri-rev}
b(\uu,\vv,\vec w)= \int \langle \nabla(\vv \otimes \uu), \vec w \rangle\,dx.
\end{equation}

\end{itemize}
\medskip

Finally, as it is now customary,  we use $A \lesssim B$ to denote the estimate $A \leq C B$ for an absolute positive constant C. 

\subsection{A general randomization setup}

Before stating the main theorem we recall a large deviation bound from \cite{BTa}  that we will use below in order to analyze the heat flow on randomized data in Section 
\ref{sec-heatfl}.

\begin{lemma}\label{largedev}[Lemma 3.1 in \cite{BTa}] Let $(l_r(\omega))_{r=1}^{\infty}$ be a sequence of real, $0$- mean, independent random  variables on a probability space $(\Omega,  A, p)$ with associated sequence of distributions $(\mu_r)_{r=1}^\infty$. Assume that $\mu_r$ satisfy the property 
\begin{equation} \label{muassum} 
\exists c > 0 \, : \, \forall \gamma \in \R, \forall r \geq 1, \, | \int_{-\infty}^\infty e^{\gamma x} \, d \mu_r(x) | \leq e^{c \gamma^2}.
\end{equation}
Then there exists $\alpha >0$ such that for every $\lambda >0$, every sequence $(c_r)_{r=1}^\infty \in \ell^2$ of  real numbers, 
\begin{equation*} 
p \left( \omega \, :\, | \sum_{r=1}^{\infty} c_r l_r(\omega) | > \lambda \right) \leq 2 e^{-\frac{\alpha \lambda^2}{\sum_r \, c_r^2}}. 
\end{equation*}
As a consequence there exists $C>0$ such that for every $q \geq 2$ and every $(c_r)_{r=1}^\infty \in \ell^2$, 
\begin{equation*}
\left\| \sum_{r=1}^{\infty} c_r l_r(\omega) \right\|_{L^q(\Omega)} \leq C \sqrt{q} \left( \sum_{r=1}^{\infty} c_r^2 \right)^{\frac{1}{2}}.
\end{equation*}  
\end{lemma}
Burq and Tzvetkov showed in \cite{BTa} that the standard real Gaussian as well as standard Bernoulli variables satisfy 
the assumption  \eqref{muassum}.

\subsection{Our randomization setup}

We now introduce the diagonal randomization of elements of $(H^s(\T^d))^d$, which we will apply to our initial data. 

\begin{definition}\label{defrand}[Diagonal randomization of elements in $(H^s(\T^d))^d$] Let 
$(l_n(\omega))_{n\in \Z^d}$ be a sequence of of real, independent, random variables on a probability space $(\Omega,  A, p)$ as in Lemma \ref{largedev}.   
 
For $\vec f\in (H^s(\T^d))^d$, let $(\vec a^i_n), \, i=1,2,\dots,d$ be its Fourier coefficients. We introduce  the map from $(\Omega,A)$ to $(H^s(\T^d))^d$ equipped with the Borel sigma algebra, defined by 
\begin{equation}\label{map}
\omega \longrightarrow \vec f^\omega, \qquad 
\vec f^\omega(x)=\left(\sum_{n\in \Z^d}l_n(\omega)\vec a_n^1e_n(x),\dots, \sum_{n\in \Z^d}l_n(\omega)\vec a_n^d e_n(x)\right),
\end{equation}
where $e_n(x)=e^{in\cdot x}$ and call such a map randomization. 
\end{definition}

By the conditions in Lemma \ref{largedev},  the map \eqref{map} is measurable and $\vec f^\omega\in L^2(\Omega; (H^s(\T^d))^d)$, is an $(H^s(\T^d))^d$-valued random variable.
Also we remark that such a randomization does not introduce any $H^s$ regularization (see Lemma B.1 in \cite{BTa} for a proof of this fact), indeed $\|\vec f^\omega\|_{H^s}\sim
\|\vec f\|_{H^s}$. However  randomization  gives improved $L^p$ 
estimates almost surely (see Proposition \ref{prop-heatfl} below). 

\begin{remark}\label{remrand} 
Since the Leray projection \eqref{div} can be written via its coordinates  
\begin{equation}\label{div1}
(\PP\vec h)_j= h_j +\sum_{k=1,\dots,d}R_jR_k h_k,
\end{equation}
where  $\widehat{R_j(\phi)}(n)=\frac{i\,n_j}{|n|}\hat{\phi}(n), \,\, \, n\in \Z^d$, 
one can easily see that the diagonal randomization defined in \eqref{map}  commutes with the Leray projection $\PP$.
\end{remark}

Having defined diagonal randomization, we can state the main results of this paper.

\subsection{Main Results }
Using the notation from Subsection \ref{subsec-not}
we introduce the following definition:
\begin{definition}\label{weak-sol}
Let $\vec f\in(H^{-\alpha}(\T^d))^d, \, \alpha>0$, $\nabla \cdot \vec f=0, $ and of mean zero\footnote{This is assumed without loss of generality. Since the mean is conserved, if it is not zero, one can replace the solution with the solution minus the mean. This new function will satisfy an equation with a first order linear modification which does not effect the estimates. }. A weak solution of the Navier-Stokes initial value problem \eqref{NS} on $[0,T]$, is a function $\vec u\in L^2_{\text{loc}}((0,T);V)\cap L^\infty_{\text{loc}}((0,T);H)\cap C_{\text{weak}}((0,T);(H^{-\alpha}(\T^d))^d)$ satisfying 
$\dfrac{d \vec u}{dt}\in L^1_{\text{loc}}((0,T),V')$ and 
\begin{eqnarray}
&&\langle \frac{d\vec u}{dt}, \vec v \rangle + ((\vec u,\vv)) + b(\vec u, \vec u,\vv)=0 \quad  \mbox{ for a.e. } t \mbox{ and for all } \vec v \in V, \label{ns-weak}\\ 
	&& \lim_{t\rightarrow 0+}\vec u(t)=\vec f \quad \mbox{weakly in  the $(H^{-\alpha}(\T^d))^d$ topology}. \label{t-limit}
\end{eqnarray}
\end{definition}

\begin{theorem}[Existence and Uniqueness in $2$D]\label{maintheorem2}
Fix $T >0$,  $0 < \alpha < \frac{1}{2}$ and let $\vec f \in (H^{-\alpha}(\T^2))^2$, $\nabla \cdot \vec f=0 $ and of mean zero. Then there exists a set $\Sigma \subset \Omega$ of  probability $1$ such that for any 
$\omega \in \Sigma$ the initial value problem  \eqref{NS} with datum $\ff^\omega$ has a unique global weak solution $\vec u$ in the sense of Definition \ref{weak-sol} of the form 
\begin{equation}\label{solutionu2} 
\vec u \, =\,  \vec u_{\ff^\omega}  \,+   \,  \vec w 
\end{equation}
where  $\vec u_{\ff^\omega} = e^{t\Delta} \ff^{\omega}$ and 
$\vec w \in  L^{\infty}([0,T]; (L^2(\T^2))^2)\cap L^2([0,T]; (\dot H^{1}(\T^2))^2).$
\end{theorem}

\medskip

\begin{theorem}[Existence in $3$D]\label{maintheorem3}
Fix $T >0$,  $0 < \alpha < \frac{1}{4}$ and let $\vec f \in (H^{-\alpha}(\T^3))^3$, $\nabla \cdot \vec f=0, $ and of mean zero. Then there exists a set $\Sigma \subset \Omega$ of  probability $1$ such that for any $\omega \in \Sigma$ the initial value problem  \eqref{NS} with datum $\ff^\omega$ has a global weak solution $\vec u$ in the sense of Definition \ref{weak-sol} of the form 
\begin{equation}\label{solutionu3} 
\vec u \, =\,  \vec u_{\ff^\omega}  \,+   \,  \vec w, 
\end{equation}
where  $\vec u_{\ff^\omega} = e^{t\Delta} \ff^{\omega}$ and 
$\vec w \in  L^{\infty}([0,T]; (L^2(\T^3))^3)\cap L^2([0,T]; (\dot H^{1}(\T^3))^3).$

\end{theorem}

\section{The heat flow on randomized data} \label{sec-heatfl} 

In this section we obtain certain a-priori estimates on the free evolution of the randomized data.
These bounds will play an important role in the proof of existence ($d=2, 3$) and uniqueness ($d=2$) in the later sections.  

\subsection{Deterministic estimates}  

\begin{lemma} \label{lemma-deterministic}
Let $0 < \alpha < 1$, $k$ a nonnegative integer and let  $\vec u_{\ff^\omega} = e^{t\Delta} \ff^{\omega}$. 
If $\ff^{\omega} \in (H^{-\alpha}(\T^d))^d$ then we have: 

\begin{eqnarray}\ 
\| \nabla^k \vec u_{\ff^\omega}  (\cdot, t) \|_{L^2_x} &\lesssim& (1 + t^{ -\frac{\alpha + k}{2} } ) \; \|\ff\|_{H^{-\alpha}}.\label{l2gradient}\\
\| \nabla^{k} \vec u_{\ff^\omega} \|_{L^{\infty}_x} &\lesssim& \left( \max \{ t^{-1}, t^{-(k+\alpha+\frac{d}{2})} \} \right)^{\frac{1}{2}} \; \|\ff\|_{H^{-\alpha}}. \label{linftygradient}
\end{eqnarray}
\end{lemma}

\begin{proof} To prove \eqref{l2gradient} we write  $\widehat{\uu}_{\ff^{\omega}}(n, t) = e^{-|n|^2 t} \widehat{\ff^{\omega}}(n)$. Then we have that: 
\begin{align*} 
	\|\nabla^k \uu_{\ff^{\omega}}  (\cdot, t) \|_{L^2_x} 
	& \sim \|  e^{-|n|^2 t} |n|^k \widehat{\ff^{\omega}}(n) \|_{l^2_n} \\
	& \lesssim  (1 + t^{ -\frac{\alpha +k }{2} } ) \; \|\ff\|_{H^{-\alpha}}, 
\end{align*} 
where to obtain the last line we used that $e^{-|n|^2 t} t^{\frac{\alpha+k}{2}}   |n|^{\alpha+k} \leq  C$.  

To prove \eqref{linftygradient} using the Fourier representation $\nabla^{k} \uu_{\ff^{\omega}}(x,t)$ we have 
\begin{align} 
	| \nabla^{k} \uu_{\ff^{\omega}} (x,t) | 
	& \leq  \sum_n  |n|^k e^{-|n|^2t} \langle n \rangle^{\alpha}  \langle n \rangle^{-\alpha}  |\widehat{\ff^{\omega}}(n) | \nonumber \\
	& \leq \|\ff\|_{H^{-\alpha}} \left( \sum_{n } \langle n \rangle^{2\alpha} |n|^{2k} e^{-2|n|^2t} \right)^{\frac{1}{2}} \nonumber \\   
	& \leq  I^{\frac{1}{2}} \; \|\ff\|_{H^{-\alpha}}, \label{lemma-det-b-introI} 
\end{align} 
where 
$$ I : = \int_{0}^{\infty} (1+\rho^2)^{\alpha} \rho^{2k} e^{-2{\rho}^2 t} \rho^{d-1} \, d\rho.$$  
In order to calculate the integral $I$ we perform the change of variables $z = \sqrt{t} \rho$ and obtain: 
\begin{equation} 
	I 
	= \int_{0}^{\infty} (1+\frac{z^2}{t})^{\alpha} \; (\frac{z}{\sqrt{t}})^{2k} \; e^{-2z^2} z^{d-1} t^{-\frac{d}{2}} \, dz \nonumber \\ 
	= I_1 + I_2, \label{lemma-det-introI12}
\end{equation}
where 
\begin{align*} 
	I_1 & =  \int_{|z| \leq \sqrt{t}} (1+\frac{z^2}{t})^{\alpha} \; (\frac{z}{\sqrt{t}})^{2k} \; e^{-2z^2} z^{d-1} t^{-\frac{d}{2}} \, dz, \\
	I_2 & =  \int_{|z| > \sqrt{t}} (1+\frac{z^2}{t})^{\alpha} \; (\frac{z}{\sqrt{t}})^{2k} \; e^{-2z^2} z^{d-1} t^{-\frac{d}{2}} \, dz. 
\end{align*}
We bound $I_1$ from above as follows:
\begin{align} 
	I_1 
	& \lesssim t^{-\frac{d}{2}} \; t^{\frac{d-2}{2}} \; \int_0^{\sqrt{t}} e^{-2z^2} z \, dz \nonumber \\ 
	& \lesssim t^{-1} (e^{-2t} - 1). \label{lemma-det-I1}  
\end{align} 
On the other hand, we bound $I_2$ from above as: 
\begin{align} 
	I_2 
	& \lesssim t^{-\frac{d}{2}} \; \int_{|z| > \sqrt{t}} \frac{ z^{2(\alpha+k)+d-1} }{ {(\sqrt{t}})^{2(\alpha+k)} } e^{-2z^2} \, dz \nonumber \\ 
	& = t^{-(k+\alpha + \frac{d}{2})} \; \int_{|z| > \sqrt{t} } z^{2(\alpha+k)+d-1}  e^{-2z^2} \, dz \nonumber \\
	& \lesssim t^{-(k+\alpha + \frac{d}{2})}. \label{lemma-det-I2}  
\end{align} 
By combining \eqref{lemma-det-I1} and \eqref{lemma-det-I2} we obtain that 
$$ 
	I \lesssim \max\{ t^{-1},  t^{-(k+\alpha + \frac{d}{2})} \}, 
$$ 
which thanks to \eqref{lemma-det-b-introI} and \eqref{lemma-det-introI12} implies the claim \eqref{linftygradient}. 
\end{proof}

\subsection{Probabilistic estimates}

\medskip
\begin{proposition} \label{prop-heatfl}
Let $T>0$ and  $ \alpha \geq 0$.  Let  $r\geq p \geq q \geq 2$,  $\sigma \ge 0$ and $\gamma \in \mathbb R$ be such that
$(\sigma + \alpha - 2\gamma) q < 2$.  Then 
there exists $C_{T}>0$ such that for every $\ff \in (H^{-\alpha}(\T^d))^d$ 
\begin{equation}\label{lp}
\|t^{\gamma} (-\Delta)^{\frac{\sigma}{2}} e^{t\Delta}\ff^\omega\|_{L^r(\Omega; L^q([0,T]; L^p_x) }
\leq C_{T} \,  \|\ff\|_{H^{-\alpha}},
\end{equation}
where $C_T$ may depend also on $p, q, r, \sigma, \gamma$ and $\alpha$. 
 
Moreover, if we set
\begin{equation}\label{setE}
E_{\lambda,T,\ff,\sigma,p}=\{\omega \in \Omega \, :\, \|t^{\gamma} (-\Delta)^{\frac{\sigma}{2}} e^{t\Delta}\ff^\omega\|_{L^q ([0,T]; L^p_x) }\geq \lambda \},
\end{equation}
then there exists $c_1, c_2>0$ such that for every $\lambda>0$ and for every $\ff\in (H^{-\alpha}(\T^d))^d$
\begin{equation}\label{probE}
P(E_{\lambda,T,\ff,\sigma,p})\leq c_1 \exp\left[{-c_2 \frac{\lambda^2}{C_T \|\ff\|^2_{H^{-\alpha}}}}\right].
\end{equation}
\end{proposition}

\begin{proof}
For  $t\ne 0$, using the notation $\vec h(x) =  \langle -\Delta \rangle^{-\frac{\alpha}{2}} \ff (x)$ and recalling the notation 
defined in \eqref{map} we have
\begin{eqnarray}
t^{\gamma} (-\Delta)^{\frac{\sigma}{2}} e^{t\Delta} \ff^\omega (x)
& = & t^{\gamma} (-\Delta)^{\frac{\sigma}{2}} \langle -\Delta \rangle^{\frac{\alpha}{2}}  e^{t\Delta} \langle -\Delta \rangle^{-\frac{\alpha}{2}} \ff^\omega (x) \nonumber \\
&=& t^{\gamma} \sum_{n\in \Z^d} |n|^\sigma (1 + |n|^2)^{\frac{\alpha}{2}} e^{-t|n|^2} \widehat {\vec {h}^{\omega}}(n) e_n(x)\nonumber \\
& \lesssim & J_{\sigma} + J_{\sigma + \alpha}, \label{probl-S12}
\end{eqnarray}
where for 
\begin{equation} \label{probl-beta} 
\beta \in \{ \sigma, \sigma + \alpha \}
\end{equation} 
we introduced $J_{\beta}$ as follows: 
$$ 
J_{\beta} =  t^{\gamma - \frac{\beta}{2}} \sum_{n\in \Z^d} t^{\frac{\beta}{2}} |n|^{\beta} e^{-t|n|^2}   \widehat {\vec {h}^{\omega}}(n) e_n(x).
$$

In order to estimate $J_{\beta}$, we observe that $t^{\frac{\beta}{2}} |n|^{\beta}  e^{-t|n|^2} \leq C$, which together with 
two applications of Minkowski's inequality, followed by an application of Lemma \ref{largedev} 
implies the first inequality in the following estimate: 
\begin{align} 
	\| J_{\beta} \|_{L^r(\Omega;  L^q( [0,T]; L^p_x)} 
	& \leq C_r \left \| \left(\sum_{n}\left| t^{\gamma - \frac{\beta}{2}} 
	 \widehat {\vec {h}}(n) e_n(x) \right|^2\right)^{\frac{1}{2}}\right \|_{L^q([0,T]; L^p_x)} \nonumber \\
	& = C_r \left\| \sum_{n} t^{2\gamma - \beta}  
	\left|  \widehat {\vec {h}}(n) e_n(x) \right|^2 \right \|^{\frac{1}{2}}_{L^{\frac{q}{2}}([0,T]; \, L^{\frac{p}{2}}_x)} \nonumber \\
	& \leq C_{r,p} \| \vec h\|_{L^2} \left( \int_0^T \left( \frac{1}{t^{\beta - 2\gamma}} \right)^{\frac{q}{2}} \right)^{\frac{1}{q}} \nonumber \\
	& = C_{r,p,q}  \| \vec h\|_{L^2} T^{\frac{1}{q} + \frac{2\gamma - \beta}{2}}, \label{afterl2} 
\end{align}
as long as 
\begin{equation} \label{probl-cond} 
(\beta - 2\gamma) \frac{q}{2} < 1, 
\end{equation} 
which is for our range \eqref{probl-beta} of $\beta$ satisfied under the assumption that $(\sigma + \alpha - 2\gamma)q < 2$. 
Now the estimate \eqref{lp} follows from \eqref{probl-S12}, \eqref{afterl2} and \eqref{probl-beta}.  

To prove estimate \eqref{probE} one uses the Bienaym\'e-Tchebishev's inequality as in Proposition 4.4 in \cite{BTa} which relies on Lemma \ref{largedev}.
\end{proof}

\section{Difference Equation and Equivalent Formulations} \label{diff-equiv}


In this section we consider two formulations of the initial value problem for the difference equation in 
\begin{equation}\label{dns0}
\left\{\begin{array}{ll}
\partial_t \ww = \Delta \ww - {\PP} \nabla (\ww \otimes \ww) + c_1 [{\PP} \nabla (\ww \otimes \vec g) + {\PP} \nabla (\vec g \otimes \ww)] +  c_2 {\PP} \nabla (\vec g \otimes \vec g) \\
\nabla \cdot \ww=0, \\
  \ww( x, 0)  =0.
  \end{array}\right.
\end{equation} 
We start with the definition of weak solutions to \eqref{dns0}.
 Again using the notation from Subsection \ref{subsec-not} 
we have:
\begin{definition}\label{diff-weak-sol}
 Assume that $\nabla \cdot \vec g=0$. A weak solution to the   initial value problem \eqref{dns0} on $[0,T]$, is a function $\vec w\in L^2((0,T);V)\cap L^\infty((0,T);H)$ satisfying 
$\dfrac{d \vec w}{dt}\in L^1((0,T); V')$ and such that for almost every $t$  and for all $  \vec v \in V$,
\begin{equation}
\langle \frac{d\vec w}{dt}, \vec v \rangle + ((\vec w,\vv)) + b(\vec w,\vec w,\vv)+ 
b(\ww,\vec g,\vec v) + b(\vec g, \ww,\vec v) + b(\vec g,\vec g,\vv) =0 \quad   \label{diff-ns-weak}
\end{equation}
and
\begin{equation}\lim_{t\rightarrow 0+}\vec w(t)=0 \quad \mbox{weakly in  the $H$ topology}. \label{diff-t-limit}
\end{equation}
\end{definition}

Now we state and prove the equivalence lemma, which is similar to the version for the Navier-Stokes equations themselves, see e.g. \cite{FLT} and Theorem 11.2 in \cite{LR}.
\begin{lemma}[The Equivalence Lemma]\label{equivalence} 
Let $T >0$, and  assume that $\nabla \cdot \vec g=0$ and 
\begin{align} 
& \| \vec g(x,t) \|_{L^2} \lesssim (1 + \frac{1}{t^{\frac{\alpha}{2}}}). \label{1-assum-g1} 
\end{align} 
Furthermore, assume that: 
\begin{equation} \label{C-L4L6}
\left\{\begin{array}{ll}
\|  \vec g\|_{L^4([0,T], L^4_x)} \leq C, \quad \mbox{ if } d = 2 \\
\| ({-\Delta})^{\frac{1}{4}} \vec  g \|_{L^2([0,T], L^6_x)}  \, + \, \| ({-\Delta})^{\frac{1}{4}} \vec  g \|_{L^{\frac{8}{3}}([0,T], L^{\frac{8}{3}}_x)} \, +\, \| \vec  g \|_{L^8([0,T], L^8_x)} \leq C, \quad \mbox{ if } d = 3,
\end{array}\right.
\end{equation}
for some $C>0$. Then the following statements are equivalent. 
\begin{enumerate}
\item[(S1)] $\ww$ is a weak solution to the  initial value problem \eqref{dns0}.
\smallskip
\item[(S2)] The function $ \ww \in L^\infty((0,T); H )  \cap L^2((0, T), V),  $  solves 
\begin{equation}\label{w-int}
	\ww(t)= -  \int_0^t e^{(t-s)\Delta} \nabla \vec F(x,s) \; ds, 
\end{equation} 
where 
\begin{equation}\label{vecF}
\vec F(x,s) =  - {\PP} (\ww \otimes \ww) +  c_1 [{\PP}  (\ww \otimes \vec g) + {\PP} (\vec g \otimes \ww)] + c_2 {\PP} (\vec g \otimes \vec g). 
\end{equation}
\end{enumerate}
\end{lemma}
\begin{proof}
We first show that {\it (S2)} implies {\it (S1)}.  Assume that $\ww$ solves the integral equation \eqref{w-int}. Define 
$$\vec W(\ww)(x,t)= -\int_0^te^{(t-s)\Delta}\nabla \vec F(x,s) \; ds.$$
Using the assumption  \eqref{1-assum-g1} on $\vec g$ and the fact that   $\ww \in L^\infty((0,T);H )  \cap L^2((0, T); V)$,  
it follows that $\vec F(x,s)\in L^1( (0,T); L^1_x)$. Hence $\nabla \vec F(x,s)\in L^1((0 ,T); \DD^{\prime}_x)$ and 
\begin{equation}e^{(t-s)\Delta}\nabla \vec F(x,s) \in  L^1( (0,T); C^\infty_x).
\label{better}
\end{equation}
We can now take the time derivative of $\vec W(\ww)(t,x)$ and easily show that
$$\partial_t \vec W(\ww)(x,t)=\Delta \vec W(\ww)(x,t)-\nabla \vec F(x,t),$$ where the equality holds in $V^{\prime}$.
Since $\ww=\vec W(\ww)$, by \eqref{w-int} and  in particular  $\lim_{t \rightarrow 0+}\ww(x,t)=0$,   we now have that $\ww$ solves \eqref{dns0}
weakly. 

\medskip
Next we show that {\it (S1)} implies {\it (S2)}. Let 
$$ \vec \Psi(x, t):= -\int_0^te^{(t-s)\Delta}\nabla \vec F(x,s) \; ds.$$ where $F(x,s)$ is as above.  Under the assumptions on $\vec g$ and $\ww$ and arguing as in the proof of
the energy estimates Theorem \ref{th-dns} near time zero, we have that $\Psi  \in L^\infty((0,T); H )  \cap L^2((0, T);  V)$ and $ \dfrac{d \Psi}{dt}   \in L^1((0, T); V').$ We then have 
$$\partial_t(\vec \Psi - \ww)=\Delta (\vec \Psi -\ww)$$ where the equality holds in $V^{\prime}$ 
and $\lim_{t\rightarrow 0+} \, (\ww(t)-\vec \Psi(t)\,) \,=\, 0$ weakly in $ H.$  A standard uniqueness 
result for the heat flow finally gives $\vec \Psi=\ww.$
\end{proof}

\bigskip

\section{Energy Estimates for the difference equation} \label{sec-estimates} 

In this section we establish energy estimates for the difference equation in \eqref{dns0}. These {\it{a priori}} estimates for $\ww$  will be used in 
Section \ref{sec-galerkin} where we construct weak solutions, see also  \cite{Tao}.

At this point it is useful to give a name to the left hand side of \eqref{energy}. We define in fact  the {\it energy} functional for $\ww$,  
\begin{equation} \label{en} 
E(\ww)(t) = \|\ww(t)\|_{L^2}^2 + c \int_0^t \int_{{\mathbb T}^d} | \nabla \otimes \ww|^2 \, dx \, ds
\end{equation} and prove the following theorem:

\begin{theorem} \label{th-dns} 
Let $T >0$, $\lambda > 0$,\,  $\gamma <0$ and $\alpha>0$ be given. Let $\vec g$ be a function such that $\nabla \cdot \vec g=0$ and 
\begin{align} 
& \| \vec g(x,t) \|_{L^2} \lesssim (1 + \frac{1}{t^{\frac{\alpha}{2}}}), \label{assum-g1} \\
& \|\nabla^k \vec g(x,t) \|_{L^{\infty}} \lesssim \left( \max \{ t^{-1}, t^{-(k+\alpha+\frac{d}{2})} \} \right)^{\frac{1}{2}}, \mbox{ for } k = 0,1\label{assum-g2} 
\end{align} 
and
\begin{equation} \label{L4L6}
\left\{\begin{array}{ll}
\| t^{\gamma} \vec g\|_{L^4([0, T];L^4_x)} \leq \lambda, \quad \mbox{ if } d = 2 \\
\| t^{\gamma} (-\Delta)^{\frac{1}{4}} \vec  g \|_{L^2([0,T];L^6_x)}  \, +\, \| t^{\gamma} (-\Delta)^{\frac{1}{4}} \vec  g \|_{L^{\frac{8}{3}}([0,T];L^{\frac{8}{3}}_x)} + \| t^{\gamma}\vec  g \|_{L^8([0,T];L^8_x)} \leq \lambda, \quad \mbox{ if } d = 3. 
\end{array}\right.
\end{equation} 
Let $ \ww \in L^\infty((0,T); (L^2(\T^d))^d )  \cap L^2((0, T);  (\dot H^1(\T^d))^d)$  be a solution to  \eqref{dns0}.  Then,  
\begin{align} 
& E(\ww)(t) \lesssim C(T,\lambda,\alpha), \quad \mbox{ for all } t \in [0, T]. \label{th-dns-enbound}  \\
& \| \frac{d}{dt} \ww \|_{L^p_t H^{-1}_x } \leq C(T,\lambda,\alpha), \label{th-dns-wrate} 
\end{align} 
with 
\begin{equation*}
p = 
\left\{\begin{array}{ll}
2, \quad \mbox{ if } d = 2 \\
\frac{4}{3}, \quad \mbox{ if } d = 3. 
\end{array}\right.
\end{equation*}
\end{theorem}
\begin{remark}In the course of the proof below we will rely on the Equivalence Lemma \ref{equivalence} and use the integral equation formulation \eqref{w-int} for $\ww$ near time zero and the other one \eqref{dns0}  away from zero. Note that \eqref{assum-g1},  \eqref{assum-g2} and \eqref{L4L6} ensure that \eqref{1-assum-g1} and \eqref{C-L4L6} hold. 
\end{remark}

\begin{proof} 
In order to prove \eqref{th-dns-enbound}, we consider two cases: $t $ near zero and $t$ away from zero. 
\medskip

\noindent {\bf{Case} ${\mathbf{t}}$ near zero:} First, thanks to the Equivalence Lemma \ref{equivalence} 
we can write $\ww$ using formulation \eqref{w-int}, 
\begin{equation}\label{w-int2}
	\ww(t)= - \int_0^t e^{(t-s)\Delta} \nabla \vec F(x,s) \; ds, 
\end{equation} 
where 
$\vec F(x,s)$ is like in \eqref{vecF}. Then, we use a continuity argument as follows. 

Assume $0 \leq \tau \leq \delta^{\ast}$, where $\delta^{\ast}$ is to be determined later. We need to proceed in different ways depending on whether
$d=2$ or $d=3$.  For $d=2$ and $\tau \in [0, \delta^{\ast}]$ 
by applying Lemma 14.1 from \cite{LR} we have: 
\begin{equation} \label{c1-wL2bound} 
	\| \ww(t) \|_{L^2_x} \lesssim \|\vec F\|_{L^2([0, \tau]; L^2_x)}, \quad \quad \mbox{ for all } t \in [0, \tau].
\end{equation} 
Also by applying the maximal regularity, see e.g. Theorem 7.3 in \cite{LR}, we obtain: 
\begin{equation} \label{c1-wH1bound} 
	\| \ww \|_{L^2 ([0, \tau]; H^1_x)} \lesssim \|\vec F\|_{L^2 ([0, \tau]; L^2_x)}.
\end{equation} 
Hence it suffices to analyze $\|\vec F\|_{L^2([0, \tau]; L^2_x)}$. We have, 
\begin{equation} \label{c1-Festimate} 
	\| \vec F\|_{L^2([0, \tau]; L^2_x)} 
	\lesssim  \|\ww \otimes \ww\|_{L^2([0, \tau]; L^2_x)} +   \|\ww \otimes \vec g\|_{L^2([0, \tau];L^2_x)} 
	+  \|\vec g \otimes \ww\|_{L^2([0, \tau]; L^2_x)} +  \|\vec g \otimes \vec g\|_{L^2([0, \tau]; L^2_x)}. 
\end{equation} 
We proceed by estimating the terms on the RHS of \eqref{c1-Festimate}. 

\medskip 

We observe that: 
\begin{equation} \label{c1-dd} 
\|\ww \otimes \ww\|_{L^2([0, \tau]; L^2_x)}   = \|\ww\|_{L^4([0, \tau]; L^4_x)}^2  \lesssim E(\ww)(\tau),
\end{equation} 
where in the last step we used that the norm $L^4_t L^4_x$ can be bounded via interpolating the two spaces 
 $L^{\infty}_t L^2_x$ and $L^2_t \dot{H}^1_x$ that appear in $E(\ww)(t)$. 

For the next two terms by H\"{o}lder's inequality we have:
\begin{align} 
	\|\ww \otimes \vec g\|_{L^2([0, \tau]; L^2_x)} + \|\vec g \otimes \vec w\|_{L^2([0, \tau]; L^2_x)} 
	& \lesssim \|\vec g\|_{L^p([0, \tau]; L^p_x)} \|\ww\|_{L^{\frac{2p}{p-2}}([0, \tau]; L^{\frac{2p}{p-2}}_x)} \nonumber \\
	& \lesssim  \|\vec g\|_{L^p([0, \tau]; L^p_x)}  \|D^{\frac{d}{p}} \ww\|_{L^{\frac{2p}{p-2}}([0, \tau]; L^2_x)} \label{c1-rd-Sob} \\
	& \lesssim  \|t^{\gamma} \vec g\|_{L^p([0, \tau]; L^p_x)} (\delta^{\ast})^{-\gamma} \|D^{\frac{d}{p}} \ww\|_{L^{\frac{2p}{p-2}}([0, \tau]; L^2_x)} (\delta^{\ast})^{\frac{p-2-d}{2p}}, \label{c1-rd-tHol} 
	 \end{align}
where to obtain \eqref{c1-rd-Sob} we used Sobolev embedding, 
and to obtain \eqref{c1-rd-tHol} we used H\"{o}lder's inequality in $t$ under the assumption that $p \geq 4$.

By letting $p=4$  in \eqref{c1-rd-tHol}, it follows from the assumptions  \eqref{L4L6} on $\vec g$, 
in conjunction with interpolation between the spaces that appear in $E(\ww)(t)$,  that 
\begin{equation}
	\|\ww \otimes \vec g\|_{L^2([0, \tau]; L^2_x)} + \|\vec g \otimes \vec w\|_{L^2([0, \tau]; L^2_x)}
	\lesssim \lambda \; (\delta^{\ast})^{-\gamma} E^{\frac{1}{2}}(\ww)(\tau), \label{c1-rd} 
\end{equation}


\medskip 

Finally thanks again to \eqref{L4L6} the last term can be estimated as
\begin{equation} \label{c1-rr}
	\|\vec g \otimes \vec g\|_{L^2([0, \tau]; L^2_x)}  =  \|\vec g\|_{L^4([0, \tau]; L^4_x)}^2  \leq \big(\lambda  (\delta^{\ast})^{-\gamma}\big)^2, 
\end{equation}
To conclude, we combine the estimates \eqref{c1-wL2bound} - \eqref{c1-Festimate} with \eqref{c1-dd}, \eqref{c1-rd} and \eqref{c1-rr} we obtain: 
\begin{equation} \label{c1-enest}
	E^{\frac{1}{2}}(\ww)(\tau) 
	\leq C_1 E(\ww)(\tau) + C_2 \lambda  (\delta^{\ast})^{-\gamma} E^{\frac{1}{2}}(\ww)(\tau) + C_3  \big(\lambda  (\delta^{\ast})^{-\gamma} \big)^2.
\end{equation} 
Hence if we denote $E^{\frac{1}{2}}(\ww)(\tau) = X$, we obtain the inequality: 
$$ 
	X \leq C_1 X^2 + C_2 \lambda   (\delta^{\ast})^{-\gamma} X + C_3 \big(\lambda  (\delta^{\ast})^{-\gamma} \big)^2.
$$ 
If we choose $\delta^{\ast}$ small enough so that $C_3  \big(\lambda  (\delta^{\ast})^{-\gamma} \big)^2 < \epsilon$ (where $\epsilon$ 
is a small absolute constant depending only on $C_1, C_2$ and $C_3$), then by a continuity argument $X$ is bounded for all
$\tau \in [0,  {\delta^{\ast}}]$. We  thus obtain 
\begin{equation} 
	E(\ww)(\tau) \leq C, \label{c1-bd}  
\end{equation} 
for all $\tau \in [0,\delta^*]$.

For $d=3$ we cannot directly close the estimates in terms of the energy $E(\ww)$. Instead we need to work with $E(\ww) + E((-\Delta)^{\frac{1}{4}} \ww) =: E_{\frac{1}{2}}(\ww)$. 
Now, for   $\tau \in [0, \delta^{\ast}]$  apply once again Lemma 14.1 from \cite{LR} we have: 
\begin{equation} \label{c1-wL2bound12} 
	\| ((-\Delta)^{\frac{1}{4}} + I )\ww(t) \|_{L^2_x} \lesssim \| ((-\Delta)^{\frac{1}{4}} + I ) \vec F\|_{L^2([0, \tau]; L^2_x)}, \quad \quad \mbox{ for all } t \in [0, \tau].
\end{equation} 
Also by applying the maximal regularity, see e.g. Theorem 7.3 in \cite{LR}, we obtain: 
\begin{equation} \label{c1-wH1bound12} 
	\| ((-\Delta)^{\frac{1}{4}} + I ) \ww \|_{L^2( [0, \tau]; H^1_x)} \lesssim \|((-\Delta)^{\frac{1}{4}} + I ) \vec F\|_{L^2([0, \tau]; L^2_x)}.
\end{equation} 
Hence it suffices to analyze $\|((-\Delta)^{\frac{1}{4}} + I )\vec F\|_{L^2([0, \tau]; L^2_x)}$. We have, 
\begin{eqnarray} \label{c1-Festimate12} 
	&&\| ((-\Delta)^{\frac{1}{4}} + I ) \vec F\|_{L^2([0, \tau]; L^2_x)}
	\lesssim  \| [((-\Delta)^{\frac{1}{4}} + I ) \ww] \otimes \ww\|_{L^2([0, \tau]; L^2_x)} \\\notag
	&+ &  \|[((-\Delta)^{\frac{1}{4}} + I )\ww] \otimes \vec g\|_{L^2([0, \tau]; L^2_x)} 
	+  \|[((-\Delta)^{\frac{1}{4}} + I )\vec g] \otimes \ww\|_{L^2([0, \tau]; L^2_x)} \\\notag
	&+&  \| [((-\Delta)^{\frac{1}{4}} + I )\vec g ]\otimes \vec g\|_{L^2([0, \tau]; L^2_x)}. 
\end{eqnarray} 
We proceed by estimating the terms on the RHS of \eqref{c1-Festimate12}. 

\medskip 

By using the same argument as in the proof of Theorem 15.2 in  \cite{LR}, we have that: 
\begin{equation} \label{c1-dd12} 
\| [((-\Delta)^{\frac{1}{4}} + I )\ww]\otimes \ww\|_{L^2([0, \tau]; L^2_x)}   \lesssim E_{\frac{1}{2}}(\ww)(\tau),
\end{equation} 

Next, by H\"{o}lder's and Sobolev's inequalities we have:
\begin{align*} 
	\| [((-\Delta)^{\frac{1}{4}} + I )\ww] \otimes \vec g\|_{L^2([0, \tau]; L^2_x)}
	&  \lesssim \|\vec g\|_{L^4([0, \tau]; L^6_x)} \|  [((-\Delta)^{\frac{1}{4}} + I ) \ww\|_{L^4([0, \tau]; L^3_x)}\\
	& \lesssim \|\vec g\|_{L^4([0, \tau]; L^6_x)} \|  \ww\|_{L^4([0, \tau]; H^1_x)}.  
	\end{align*}
Now we note that $ \|  \ww\|_{L^4([0, \tau]; H^1_x)} \lesssim [E_{\frac{1}{2}}(\ww)]^{\frac{1}{2}}(\tau)$ by interpolation. Hence 
\begin{equation}\label{c1-rd-tHol12}
\| [((-\Delta)^{\frac{1}{4}} + I )\ww] \otimes \vec g\|_{L^2([0, \tau]; L^2_x)}  \lesssim  \|t^{\gamma} \vec g\|_{L^8([0, \tau]; L^8_x)} (\delta^{\ast})^{-\gamma +\frac{1}{8}} [E_{\frac{1}{2}}(\ww)]^{\frac{1}{2}}.
 \end{equation}
	
For the next term we have 
\begin{eqnarray*}
	\|\ww \otimes  [((-\Delta)^{\frac{1}{4}} + I )\vec g] \|_{L^2([0, \tau]; L^2_x)}  &\lesssim& \| ((-\Delta)^{\frac{1}{4}} + I )\vec g \|_{L^2([0, \tau]; L^6_x)} \| \ww \|_{L^\infty([0, \tau]; L^3_x)} \\ 
	&\lesssim&  (\delta^{\ast})^{-\gamma} \| t^\gamma ((-\Delta)^{\frac{1}{4}} + I )\vec g \|_{L^2([0, \tau]; L^6_x)} \| \ww \|_{L^\infty([0, \tau]; H^{\frac{1}{2}}_x)}.\end{eqnarray*}
Hence 
\begin{equation}\label{c1-rd-tHol12-2}
\| \ww \otimes  [((-\Delta)^{\frac{1}{4}} + I )g]\|_{L^2([0, \tau]; L^2_x)} \lesssim   (\delta^{\ast})^{-\gamma} \| t^\gamma ((-\Delta)^{\frac{1}{4}} + I )\vec g \|_{L^2([0, \tau]; L^6_x)} [E_{\frac{1}{2}}(\ww)]^{\frac{1}{2}}.
 \end{equation}

\medskip 

Finally thanks again to \eqref{L4L6} the last term can be estimated as
\begin{equation} \label{c1-rr12}
	\|[((-\Delta)^{\frac{1}{4}} + I )\vec g] \otimes \vec g\|_{L^2([0, \tau]; L^2_x)} \lesssim  \|((-\Delta)^{\frac{1}{4}} + I )\vec g\|_{L^\frac{8}{3}([0, \tau]; L^\frac{8}{3}_x)} \|\vec g\|_{L^8([0, \tau]; L^8_x)}  \leq \big(\lambda  (\delta^{\ast})^{-\gamma}\big)^2, 
\end{equation}

To conclude, we combine the estimates \eqref{c1-wL2bound12} - \eqref{c1-Festimate12} with \eqref{c1-dd12}, \eqref{c1-rd-tHol12}, \eqref{c1-rd-tHol12-2} and \eqref{c1-rr12} to obtain: 
\begin{equation} \label{c1-enest}
	[E_{\frac{1}{2}}(\ww)(\tau)]^{\frac{1}{2}} 
	\leq C_1 E_{\frac{1}{2}}(\ww)(\tau) + C_2 \lambda  (\delta^{\ast})^{-\gamma + \frac{1}{8}} [E_{\frac{1}{2}}(\ww)(\tau)]^{\frac{1}{2}} + C_3  \big(\lambda  (\delta^{\ast})^{-\gamma} \big)^2.
\end{equation} 
Hence if we denote $[E_{\frac{1}{2}}(\ww)(\tau)]^{\frac{1}{2}} = X$, we obtain the inequality: 
$$ 
	X \leq C_1 X^2 + C_2 \lambda   (\delta^{\ast})^{-\gamma + \frac{1}{8}} X + C_3 \big(\lambda  (\delta^{\ast})^{-\gamma} \big)^2.
$$ 
If we choose $\delta^{\ast}$ small enough so that $C_3  \big(\lambda  (\delta^{\ast})^{-\gamma} \big)^2 < \epsilon$ (where $\epsilon$ 
is a small absolute constant depending only on $C_1, C_2$ and $C_3$), then by a continuity argument $X$ is bounded for all
$\tau \in [0,  {\delta^{\ast}}]$. We  thus obtain 
\begin{equation} 
	E(\ww)(\tau) \lesssim E_{\frac{1}{2}}(\ww)(\tau) \leq C, \label{c1-bd12}  
\end{equation} 
for all $\tau \in [0,\delta^*]$.

\medskip 

\noindent {\bf{Case}} ${\mathbf{t \in [\delta^*, T]:}}$ By the standard energy argument for \eqref{dns0} we have, 
\begin{align}
	\frac{d}{dt} E(\ww)(t) 
	& = \int_{{\T}^d} 2 \ww(x,t) \cdot \ww_t(x,t) \, dx + 2 \int_{{\T}^d} |\nabla \otimes \ww|^2(x,t) \, dx \nonumber \\ 
	& = \int_{{\T}^d} 2 \ww\,  \Delta \ww \, dx - 2 \int_{{\T}^d} \ww \cdot \PP \nabla (\ww \otimes \ww) \, dx + 2 \int_{{\T}^d} |\nabla \otimes \ww|^2 \, dx \label{c2-difen-ns} \\ 
	& + 2c_1  \left( \int_{{\T}^d} \ww \cdot \PP \nabla (\ww \otimes \vec g) \, dx  + \int_{{\T}^d} \ww \cdot \PP \nabla (\vec g \otimes \ww) \, dx \right) 
	   + 2c_2 \int_{{\T}^d} \ww \cdot \PP \nabla (\vec g \otimes \vec g) \, dx. \label{c2-difen-rest} 
\end{align}
Now we observe that 
the expression in \eqref{c2-difen-ns} equals zero as in the case of solutions to the Navier-Stokes equations itself. It remains to estimate \eqref{c2-difen-rest}. 
In order to do that we first note that since $\vec g$ is divergence-free, 
$\int_{{\T}^d} \ww \cdot \PP \nabla (\ww \otimes \vec g) \, dx = \int_{{\T}^d} \ww \cdot \PP (\vec g \cdot \nabla) \ww \, dx$ and the last expression equals zero 
thanks to the skew-symmetry property. Also since $\ww$ is divergence-free too, we observe that:  
\begin{equation}  \label{c2-difen-rest1} 
\int_{{\T}^d} \ww \cdot \PP \nabla (\vec g \otimes \ww) \, dx =  \int_{{\T}^d} \ww \cdot \PP (\ww \cdot \nabla \vec g) \, dx \lesssim \|\ww\|_{L^2_x}^2 \| \nabla \vec g\|_{L_x^{\infty}}.
\end{equation} 
On the other hand by H\"{o}lder's inequality, 
\begin{equation} 
	\int_{{\T}^d} \ww \cdot \PP \nabla (\vec g \otimes \vec g) \, dx 
	\leq \|\ww\|_{L^2_x} \|\vec g\|_{L^2_x} \| \nabla \vec g\|_{L_x^{\infty}} \label{c2-difen-rest2}. 
\end{equation} 
Now by combining \eqref{c2-difen-rest}, \eqref{c2-difen-rest1} and \eqref{c2-difen-rest2} and 
using the assumptions \eqref{assum-g1} and \eqref{assum-g2} we obtain: 
\begin{align} 
	\frac{d}{dt} E(\ww)(t) 
	& \lesssim  E(\ww)(t) \| \nabla \vec g\|_{L_x^{\infty}} + E^{\frac{1}{2}}(\ww)(t)  \|\vec g\|_{L^2_x} \| \nabla \vec g\|_{L_x^{\infty}} \nonumber \\
	& \lesssim  h(t)  E(\ww)(t) + m(t) E^{\frac{1}{2}}(\ww)(t) \label{c2-usingprob}
\end{align} 
with 
\begin{align*} 
	h(t) & = \left( \max \{ t^{-1}, t^{-(1+\alpha+\frac{d}{2})} \} \right)^{\frac{1}{2}}, \\
	m(t) & = (1 + \frac{1}{t^{\frac{\alpha}{2}}} ) \left( \max \{ t^{-1}, t^{-(1+\alpha+\frac{d}{2})} \} \right)^{\frac{1}{2}}. 
\end{align*} 
Furthermore, using the above expressions for $h(t)$ and $m(t)$ we get: 
\begin{eqnarray} \label{c2-h} 
	\int_{\delta^*}^T h(t) \, dt & =& \int_{\delta^*}^{1} h(t) \, dt + \int_{1}^T h(t) \; dt \\
	&=& \int_{\delta^*}^{1} \frac{1}{ t^{\frac{1}{2}(1+\alpha +\frac{d}{2})} } \, dt + \int_{1}^T \frac{1}{t^{\frac{1}{2}}} \; dt  
	\lesssim  \,  (\delta^{*})^{\frac{1}{2}-\frac{\alpha}{2}  - \frac{d}{4}} + T^{\frac{1}{2}}, \nonumber 
\end{eqnarray} 
and 
\begin{eqnarray} \label{c2-m} 
	\int_{\delta^*}^T m(t) \, dt 
	& =& \int_{\delta^*}^{1} m(t) \, dt + \int_{1}^T m(t) \; dt  \\
	& = &\int_{\delta^*}^{1} \frac{1}{ t^{\frac{1}{2}+\alpha +\frac{d}{4}} } \, dt + \int_{1}^T \frac{1}{t^{\frac{1+\alpha}{2}}} \; dt 
	\lesssim \, (\delta^{*})^{\frac{1}{2} - \alpha  - \frac{d}{4}} + T^{\frac{1-\alpha}{2}}. \nonumber
\end{eqnarray} 
Therefore by combining \eqref{c2-usingprob}, \eqref{c2-h} and \eqref{c2-m} we obtain the bound
\begin{equation} \label{c2-bd} 
	E(\ww)(t) \leq C(T, \delta^*, \alpha)  
\end{equation}
for all $t \in [\delta^*, T]$. 
Now \eqref{th-dns-enbound} follows from \eqref{c1-bd} and \eqref{c2-bd}.

\medskip
To prove \eqref {th-dns-wrate} we  let 
\begin{equation*}
p = 
\left\{\begin{array}{ll}
2, \quad \mbox{ if } d = 2 \\
\frac{4}{3}, \quad \mbox{ if } d = 3. 
\end{array}\right.
\end{equation*}

Since $\vec w$ satisfies \eqref{dns0} we observe that: 
\begin{align} \label{th-dns-rate-bd} 
	 \| \frac{d}{dt}\ww\|_{L^p([0,T]; H^{-1}_x)} 
	 &\lesssim \| \Delta \ww\|_{L^p([0,T]; H^{-1}_x)}  + \| \nabla (\ww \otimes \ww)\|_{L^p([0,T]; H^{-1}_x)}  \\
	 &+  \| \nabla (\ww \otimes \vec g)\|_{L^p([0,T]; H^{-1}_x)}
	+ \| \nabla (\vec g \otimes \ww)\|_{L^p([0,T]; H^{-1}_x)}   \\ 
	&+ \| \nabla (\vec g \otimes \vec g)\|_{L^p([0,T]; H^{-1}_x)}.     \nonumber
\end{align}
We estimate the first term on the RHS of \eqref{th-dns-rate-bd} as follows: 
\begin{equation} \label{th-dns-rate-1}
	\| \Delta \ww\|_{L^p([0,T]; H^{-1}_x)} \lesssim 
	\left\{\begin{array}{ll}
	\|\nabla \ww\|_{L^2([0,T]; L^2_x)}, \quad \mbox{ if } d = 2 \\
	\|\nabla \ww\|_{L^{\frac{4}{3}}([0,T]; L^2_x)} \lesssim T^{\frac{1}{4}} \|\nabla \ww\|_{L^2([0,T]; L^2_x)}, \quad \mbox{ if } d = 3. 
\end{array}\right.
\end{equation}
To estimate the second term on the RHS of  \eqref{th-dns-rate-bd}  we notice that for $d=2$
\begin{equation}  \label{th-dns-rate-2d2}
	\| \nabla (\vec w \otimes \ww)\|_{L^p([0,T]; H^{-1}_x)} \lesssim  \left\| \ww \otimes \ww \right\|_{L^2([0,T]; L^2_x)} 
	 \leq \|\ww\|_{L^4([0,T]; L^4_x)}^2 \lesssim E(\ww)(T),  
\end{equation} 
while for $d=3$ we use Gagliardo-Nirenberg's inequality to  obtain:  
\begin{equation} \label{th-dns-rate-2d3}
	\| \nabla (\ww \otimes \ww)\|_{L^p([0,T]; H^{-1}_x)}	
	 \lesssim \left\| \, \|\ww\|_{L^2_x}^{\frac{1}{2}} \|\ww\|_{H^1_x}^{\frac{3}{2}} \, \right\|_{L^{\frac{4}{3}}([0,T])} 
	 	 \lesssim  \|\ww\|_{L^{\infty}([0,T]; L^2_x)}^{\frac{1}{2}} \|\ww\|_{L^2([0,T]; H^1_x)}^{\frac{3}{2}} \nonumber\\ 
        \lesssim E(\ww)(T). 
\end{equation} 
To estimate the third term on the RHS of  \eqref{th-dns-rate-bd}, for $d=2$, we proceed as follows: 
\begin{align}
	\| \nabla (\ww \otimes \vec g)\|_{L^p([0,T]; H^{-1}_x)}^2   	
	& = \int_0^T \| \nabla (\ww \otimes \vec g)\|_{H^{-1}_x}^2 \, ds  \nonumber \\
	& = \int_0^T \| \ww \otimes \vec g\|_{L^2_x}^2 \, ds \nonumber \\ 
	& \leq \|\vec g\|_{L^4([0,T]; L^4_x)}^2 \|\ww\|_{L^4([0,T]; L^4_x)}^2 \label{th-dns-rate-3d2-inter} \\
	& \leq \left(\lambda T^{-\gamma}\right)^2  \|\ww\|_{L^4([0,T]; L^4_x)}^2 \label{th-dns-rate-3d2a} \\
	& \leq C(T,\lambda) \; E(\ww)(T), \label{th-dns-rate-3d2}  	
\end{align}
where to obtain \eqref{th-dns-rate-3d2a} we used the assumption \eqref{L4L6} on $\vec g$.
The fourth term can be estimated analogously for $d=2$. On the other hand to estimate the third term for $d=3$ we have,  

\begin{align}  \| \nabla (\ww \otimes \vec g)\|_{L^p([0,T]; H^{-1}_x)}^2 	&  \leq  \| \ww \otimes \vec g \|_{L^{\frac{4}{3}}([0,T]; L^2_x)}^2  \nonumber \\
& \leq \|\vec g\|_{L^2([0,T]; L^6_x)}^2  \|\vec w\|_{L^4([0,T]; L^3_x)}^2  \nonumber \\
&\leq C_T  \left(\lambda T^{-\gamma}\right)^2  \|t^\gamma \vec g\|_{L^8([0,T]; L^8_x)}^2  \|\vec w\|_{L^4([0,T]; H^{\frac{1}{2}}_x)}^2  \nonumber \\
& \leq C(T, \lambda) \; E(\ww)(T). \label{th-dns-rate-3d3} 
\end{align}

Also the fourth term for $d=3$ can be estimated analogously. 

Finally, in order to estimate the fifth term on the RHS of  \eqref{th-dns-rate-bd}, for $d=2$, we proceed in a similar way as when we 
estimated the third term for $d=2$: 
\begin{align}
	\| \nabla (\vec g \otimes \vec g)\|_{L^p([0,T]; H^{-1}_x)}^2 	
	& = \int_0^T \| \nabla (\vec g \otimes \vec g)\|_{H^{-1}_x}^2 \, ds \nonumber \\
	& = \int_0^T \| \vec g \otimes \vec g\|_{L^2_x}^2 \, ds \nonumber \\ 
	& \leq \|\vec g\|_{L^4([0,T]; L^4_x)}^4 \nonumber \\
	& \leq \left(\lambda T^{-\gamma}\right)^4, \label{th-dns-rate-5d2}
\end{align}
where to obtain the last line we used the assumptions of the theorem.
On the other hand, to estimate the fifth term for $d=3$ we  have 
\begin{align}
	\| \nabla (\vec g \otimes \vec g)\|_{L^p([0,T]; H^{-1}_x)}^2 	
	& \leq T ^{\frac{1}{2}} \| \nabla (\vec g \otimes \vec g)\|_{L^2_t H^{-1}_x}^2  	\nonumber \\
	& \leq T ^{\frac{1}{2}} \int_0^T \| \vec g \otimes \vec g\|_{L^2_x}^2 \, ds \nonumber \\ 
	& \leq T ^{\frac{1}{2}} \|\vec g\|_{L^4([0,T]; L^4_x)}^2  \nonumber \\
	& \leq C_T  \|\vec g\|_{L^8([0,T]; L^8_x)}^2 \nonumber \\
	& \leq C_T (\lambda T^{-\gamma})^2.  \label{th-dns-rate-5d3}
\end{align} 
Collecting the above estimates we obtain \eqref{th-dns-wrate}. 

\end{proof}

\section{Construction of weak solutions to the difference equation} \label{sec-galerkin}

In this section we construct weak solutions to the initial value problem \eqref{dns0}.  We denote the spatial Fourier tranform of $\vec f$ as
$$ 
	\ft{{\vec f}}(\bk,t) = \int_{\T^d} {\vec f}(x,t) e^{- i \bk \cdot x} dx,
$$
with inverse transform 
$$ 
	\vec f(x,t)   = \sum_{\bk} \ft{\vec f}(\bk ,t) e^{i \bk \cdot x},
$$ 
where $\bk$ represents the discrete wavenumber: 
$$ 
	\bk = \sum_{j=1}^d (2 \pi n_j) {\be}_j, \quad \quad n_j \in {\mathbb Z}, 
$$
and ${\be}_j$ is the unit vector in the $j$-th direction.  By $P_M$ we denote the rectangular Fourier projection operator: 
$$ 
	P_M \vec f = \sum_{ \{\bk \, : \, |n_j| \leq M \, \text{ for } 1 \leq j \leq d \} } \ft{\vec f}(\bk) e^{i \bk \cdot x}.   $$	

Note $P_M$ is a bounded operator in $L^p(\mathbb T^d)$ for $1 < p < \infty$.

\begin{theorem}\label{construction-weak}
Let $T >0$, $\lambda > 0$,\,  $\gamma <0$ and $\alpha>0$ be given. 
Assume that the function $\vec g$ satisfies $\nabla \cdot \vec g = 0$ and
\begin{align} 
& \| \vec g(x,t) \|_{L^2} \lesssim (1 + \frac{1}{t^{\frac{\alpha}{2}}}) \label{2-assum-g1} \\
& \|\nabla^k  P_M \vec g(x,t) \|_{L^{\infty}} \lesssim \left( \max \{ t^{-1}, t^{-(k+\alpha+\frac{d}{2})} \} \right)^{\frac{1}{2}} \mbox{ for } k =0,1.  \label{2-assum-g2}
\end{align} 
Furthermore, assume that we have: 
\begin{equation} \label{lambda-L4L6}
\left\{\begin{array}{ll}
\| t^{\gamma} \vec g\|_{L^4([0, T];L^4_x)} \, \leq \, \lambda, \quad \mbox{ if } d = 2 \\
\| t^{\gamma} (-\Delta)^{\frac{1}{4}} \vec  g \|_{L^2([0,T];L^6_x)}  \, +\, \| t^{\gamma} (-\Delta)^{\frac{1}{4}} \vec  g \|_{L^{\frac{8}{3}}([0,T];L^{\frac{8}{3}}_x)} + \| t^{\gamma}\vec  g \|_{L^8([0,T];L^8_x)} \leq \lambda, \quad \mbox{ if } d = 3. 
\end{array}\right.\end{equation}
Then there exists a weak solution $\ww$ for the  initial value problem \eqref{dns0} in the sense of Definition \ref{diff-weak-sol}.
\end{theorem}

\begin{remark} Since $P_M$ is a bounded operator in $L^p$ for $1 < p < \infty$,  
$P_Mg$ satisfies \eqref{2-assum-g1} and  \eqref{lambda-L4L6} as  $g$ itself. So from this point on we will not make a distinction between $P_Mg$ and $g$.

Also, to keep the notation light, in the proof of Theorem \ref{construction-weak} we shall consider the initial value problem \eqref{dns0} 
with $c_1 = c_2 = 1$.
\end{remark}

\begin{proof} 
In the construction of weak solutions, we follow in part the approach based on Galerkin approximations from Chapter 5 of 
Doering and Gibbon \cite{DG} and from Chapter 8 of Constantin and  Foias \cite{CF}.  From now on we drop the vector notation to keep the notation light.   The plan is to construct a global weak solution via finding its Fourier coefficients, which, in turn, will be achieved  by solving finite dimensional ODE systems for them.
To determine the ODE systems we 
start by formally applying the Fourier transform to the difference equation \eqref{dns0}.
 \begin{align} 
	\frac{d}{dt} \ft{w}(\bk,t) 
	& = - k^2 \ft{w}(\bk,t)  \label{Fdns2} \\
	& + i \left( I - \frac{\bk \bk^{T}}{k^2} \right) \sum_{\bk'+\bk'' = \bk} \Big(
	\ft{w}(\bk',t) \cdot \bk'' \ft{w}(\bk'',t)
	+ \ft{w}(\bk',t) \cdot \bk'' \ft{g}(\bk'',t) \nonumber \\
	& \quad \quad \quad \quad + \ft{g}(\bk',t) \cdot \bk'' \ft{w}(\bk'',t)
	+ \ft{g}(\bk',t) \cdot \bk'' \ft{g}(\bk'',t) \Big), \nonumber \\
	\bk \cdot \ft{w}(\bk,t) & = 0, \label{F-dns2-div}  
\end{align} 
where  $\ft{w} ({\bf{0}} , t ) =0$,  and  $I - \frac{\bk \bk^{T}}{k^2}$ is the projection onto the divergence-free vector fields in Fourier  space. Here $I$ is the unit tensor and 
$| \bk|= k$.  

As in \cite{DG},  we now introduce the Galerkin approximations as truncated Fourier expansions. More precisely, 
let $M$ be a positive integer and consider 
\begin{equation} \label{k-trunc} 
	\bk = \sum_{j=1}^d (2 \pi n_j) {\be}_j, \quad \quad n_j = \pm 1, \pm 2, ..., \pm M.
\end{equation}

We look for the complex variables  $\ft{w^M}(\bk,t)$, $\bk$ as in \eqref{k-trunc},  solving the following finite system of ODE: 
\begin{eqnarray} \label{GFdns2}
	&&\frac{d}{dt} \ft{w^M}(\bk,t) 
	 = - k^2 \ft{w^M}(\bk,t)    \\\notag
	& + & i \left( I - \frac{\bk \bk^{T}}{k^2} \right) \sum_{\bk'+\bk'' \stackrel{M}{=} \bk} \Big(\ft{w^M}(\bk',t) \cdot \bk'' \ft{w^M}(\bk'',t)\\\notag
	&+& \ft{w^M}(\bk',t)
	 \cdot \bk'' \ft{g}(\bk'',t)   
	 +  \ft{g}(\bk',t) \cdot \bk'' \ft{w^M}(\bk'',t) + \ft{g}(\bk',t) \cdot \bk'' \ft{g}(\bk'',t) \Big),\\ \label{GF-dns2-div} 
	&& \bk \cdot \ft{w^M}(\bk,t)  = 0,  
\end{eqnarray} 
where the sum over $\bk'$ extends over the range where both $\bk'$ and $\bk''$ are as in \eqref{k-trunc}.
The system \eqref{GFdns2} - \eqref{GF-dns2-div} is considered  with zero initial data.
 
Let $T>0$ be fixed. We now show that for any fixed $M>0$, the ODE system \eqref{GFdns2} admits a unique solution in  $X_T:=C([0,T], \ell^2)\cap L^2([0,T], \bk\ell^2)$. 
We proceed by a fixed point argument.  Define 
\begin{align*} 
	\Phi( \ft{w^M})(\bk,t) \, 
	& :=  \, - \int_0^t    k^2 \ft{w^M}(\bk,s)  ds \\
	& + \int_0^t   \Big[ i \left( I - \frac{\bk \bk^{T}}{k^2} \right) \sum_{\bk'+\bk'' \stackrel{M}{=} \bk} \Big(\ft{w^M}(\bk',s) \cdot \bk'' \ft{w^M}(\bk'',s)+ \ft{w^M}(\bk',s) \cdot \bk'' \ft{g}(\bk'',s) \\ 
	& \quad \quad \quad \quad + \, \ft{g}(\bk',s) \cdot \bk'' \ft{w^M}(\bk'',s) + \ft{g}(\bk',s) \cdot \bk'' \ft{g}(\bk'',s) \Big) \Big] \, ds.  
\end{align*}

Assume  $0 \leq t <T$. Let $\delta$ such that $0 < \delta < T$ be determined later. For $t \in [0, \delta]$, we have
\begin{eqnarray} \label{close0}
 \Vert  \Phi( \ft{w^M})(\cdot,t) \Vert_{\ell^2} &\lesssim&    M^2 \, \delta\, \Vert \ft{w^M}\Vert_{L_t^{\infty}\ell^2} + \delta M^{1+\frac{d}{2}} \Vert \ft{w^M}\Vert^2_{L_t^{\infty}\ell^2} \\
 &+& M^{1+\frac{d}{2}+} \, \delta^{1- \frac{\alpha}{2}}  \, \Vert \ft{w^M}\Vert_{L_t^{\infty}\ell^2} + M \delta^{h(d)-2\gamma} \lambda^2, \nonumber 
 \end{eqnarray}
with 
\begin{equation*}
h(d) = 
\left\{\begin{array}{ll}
\frac{1}{2}, \quad \mbox{ if } d = 2 \\
\frac{3}{4},\quad \mbox{ if } d = 3.
\end{array}\right.
\end{equation*}
To obtain the second term on the RHS above we used Plancherel and the  Sobolev's embedding for the $L^4_x$ norm. To obtain the third term,  
we used Plancherel, \eqref{2-assum-g1} and the Sobolev embedding for the $L^{\infty}_x$ norm of  $\mathcal F^{-1}( \ft{w^M})$, the inverse Fourier transform of $\ft{w^M}$. 
Finally for the fourth term we used Plancherel and \eqref{lambda-L4L6}. 
In a similar manner we can also show that 
\begin{eqnarray} \label{close1}
 \Vert  \bk\Phi( \ft{w^M})(\cdot,t) \Vert_{L^2([0,\delta],\ell^2)} &\lesssim&    M^3 \, \delta^\frac{3}{2}\, \Vert \ft{w^M}\Vert_{L_t^{\infty}\ell^2} + \delta^\frac{3}{2} M^{2+\frac{d}{2}+} \Vert \ft{w^M}\Vert^2_{L_t^{\infty}\ell^2} \\
 &+& M^{2+\frac{d}{2}+} \, \delta^{-\gamma+ \theta(d)}  \, \Vert \ft{w^M}\Vert_{L_t^{\infty}\ell^2}\, \lambda \, + \, M^2 \delta^{\rho(d)-2\gamma} \lambda^2, \nonumber 
 \end{eqnarray}
with 
\begin{equation*}
\theta(d) = 
\left\{\begin{array}{ll}
\frac{5}{4}, \quad \mbox{ if } d = 2 \\
\frac{11}{8},\quad \mbox{ if } d = 3
\end{array}\right. \quad \rho(d) = 
\left\{\begin{array}{ll}
1, \quad \mbox{ if } d = 2 \\
\frac{5}{4},\quad \mbox{ if } d = 3.
\end{array}\right.   
\end{equation*}
If we let $R = 2 M$ and $\delta= \delta(M, \lambda)$ small enough, we have the $\Phi$ maps
balls of radius $R$ in $X_\delta$ to themselves continuously. A similar argument shows $\Phi$ is also a contraction and as a consequence the ODE system \eqref{GFdns2} 
has a unique solution $\ft{w^M}$ in $X_\delta$. Therefore by applying Plancherel's Theorem we conclude that the function $w^M (x, t)$, given by the inverse Fourier transform of $\{ \ft{w^M}(\bk,t) \}_{\bk}$, belongs to  
$L^{\infty}([0, \delta]; (L^2({\mathbb T}^d))^d) \cap L^2 ([0, \delta];  (\dot{H}^1({\mathbb T}^d))^d)$.

We next note that in $[0, \delta]$ the function $w^M (x, t)$ is a solution to 
the following system: 
\begin{equation} \label{dns2-proj}
\left\{\begin{array}{ll}
	 \partial_t w^M 
	= - \Delta w^M - P_M \Big[ {\PP} \nabla (w^M \otimes w^M) + {\PP} \nabla (w^M \otimes P_M g)  \\
	 \quad \quad \quad \quad \quad \quad + {\PP} \nabla (P_M g \otimes w^M) + {\PP} \nabla (P_M g \otimes P_M g)\Big]  \\
	 \nabla \cdot w^M  = 0 \\
	 w^M(x,0)  = 0, 
\end{array}	 
\right.
\end{equation}

Since $P_Mg$ satisfies the
same assumptions as $g$ in  Section \ref{sec-estimates}, we can repeat the proof  of Theorem \ref{th-dns} to 
conclude that  $w^M (x, t)$ is in  $L^{\infty}([0, \delta]; H) \cap L^2 ([0, \delta]; V)$ and satisfies the energy bounds given by \eqref{th-dns-enbound}.
As a consequence we can use an iteration argument to advance the solution of  \eqref{GFdns2} up to time $T$.

Therefore the function $w^M (x, t)$ given by the inverse Fourier transform of $\{ \ft{w^M}(\bk,t) \}_{\bk}$ is  
in  $L^{\infty}((0, T); H) \cap L^2 ((0, T);  V)$ and it satisfies the energy bound given by \eqref{th-dns-enbound} and  \eqref{th-dns-wrate} in Theorem \ref{th-dns}.
Now by applying a standard compactness argument, together with the fact that $P_Mg$ converges strongly to $g$ in  $L^{p}$, one obtains  a weak 
solution $w$ to \eqref{dns0}  on $[0, T]$, (see also Chapter 8 of \cite{CF}). Since $T$ was arbitrary large, we obtained a global weak solution. 
\end{proof}

\section{Uniqueness in 2D} \label{sec-uniq} 

In this section we present a uniqueness result for solutions of the  initial value problem \eqref{dns0}
in $d=2$. The result can be stated as follows:

\begin{theorem} \label{lemma-uniq}
Assume that $\vec g$ satisfies the same conditions as in Theorem \ref{construction-weak}. 
Then in $d=2$ any two weak solutions  to  \eqref{dns0}  in $L^2([0,T]; V) \cap L^{\infty}((0, T); H)$ coincide. 
\end{theorem} 

\begin{proof} 
Our proof is inspired by the proof of Theorem 10.1 in Constantin-Foias \cite{CF}, which establishes a 
related uniqueness result for solutions to the Navier-Stokes equations. 

Let $w_j$ with $j=1,2$ be two solutions of \eqref{dns0} with $\vec g$ satisfying  \eqref{2-assum-g1} -  \eqref{lambda-L4L6}. Let $\vv = \ww_1 - \ww_2$. 
Then $\vv$ satisfies the equation:
\begin{equation}\label{eq-v}
\left\{\begin{array}{ll}
\partial_t \vv = \Delta \vv - {\PP} \nabla (\ww_1 \otimes \vv) - {\PP} \nabla (\ww_2 \otimes \vv) + c_1 \left[ {\PP} \nabla (\vv \otimes \vec g) + {\PP} \nabla (\vec g \otimes \vv) \right]\\
\nabla \cdot \vv=0, \\
 \vv( x, 0)  = 0.
\end{array}\right.
\end{equation}

By pairing in $(L^2(\T^2))^2$ the first equation in \eqref{eq-v} with $\vv$ we obtain: 
\begin{eqnarray*}
	( \frac{d}{dt} \vv, \vv ) &=& (\Delta \vv, \vv ) - \int  {\PP} \nabla (\ww_1 \otimes \vv) \cdot \vv \, dx -  \int {\PP} \nabla (\ww_2 \otimes \vv) \cdot \vv \, dx \\
	&&
	                                                 +  c_1 \int [{\PP} \nabla (\vec g \otimes \vv) \cdot \vv  + {\PP}\nabla ( \vv\otimes \vec g) \cdot \vv ]\, dx, 
\end{eqnarray*}
which thanks to the equality $\nabla (\ww_j \otimes \vv) = (\vv \cdot \nabla) \ww_j$ (that is valid for $j=1,2$ since each $w_j$ is divergence free and hence 
$v$ is divergence free too) becomes:  
\begin{eqnarray} 
	( \frac{d}{dt} \vv, \vv ) &= &(\Delta \vv, \vv) - \int  {\PP} ((\vv \cdot \nabla) \ww_1) \cdot \vv \, dx -  \int {\PP} ((\vv \cdot \nabla) \ww_2) \cdot \vv \, dx \\
	                                   &&              +  c_1 \int [{\PP} \nabla (\vec g \otimes \vv) \cdot \vv + {\PP}\nabla ( \vv\otimes \vec g) \cdot \vv ]\, dx.\notag
\end{eqnarray} 	                                                 
Therefore after performing integration by parts in the last term of the above expression, and using H\"older's inequality to bound the last three terms on the RHS we obtain: 
\begin{align} 
	& \frac{d}{dt} \|\vv(t)\|_{L^2_x}^2 + \|\nabla \vv\|_{L^2_x}^2 \nonumber \\ 
	& \quad \quad \leq \| \nabla \ww_1 \|_{L^2_x} \| \vv\|_{L^4_x}^2 + \|\nabla \ww_2 \|_{L^2_x} \| \vv\|_{L^4_x}^2 + c_1 \| \vec g \|_{L^4_x}  \| \vv\|_{L^4_x} \|\nabla \vv\|_{L^2_x} \nonumber \\
	& \quad \quad \leq \| \nabla \ww_1 \|_{L^2_x} \| \vv\|_{L^4_x}^2 + \| \nabla \ww_2 \|_{L^2_x} \| \vv\|_{L^4_x}^2 + c_1 \left( \frac{1}{\mu} \| \vec g \|_{L^4_x}^2  \| \vv\|_{L^4_x}^2  + \mu \|\nabla \vv\|_{L^2_x}^2 \right) \label{uniq-gv-Y1} \\
	& \quad \quad \leq \left( \sum_{j=1}^2 \| \nabla \ww_j \|_{L^2_x} + c_1 \frac{1}{\mu} \| \vec g \|_{L^4_x}^2 \right)  \|\vv\|_{L^2_x} \|\nabla \vv\|_{L^2_x} +  c_1 \mu \|\nabla \vv\|_{L^2_x}^2 \label{uniq-afterinterp} \\
	& \quad \quad \leq \left( \sum_{j=1}^2 \frac{1}{\nu_j} \|\nabla \ww_j\|_{L^2_x}^2 + \frac{c_1^2}{\mu^2 \nu_3} \| \vec g\|_{L^4_x}^4 \right) \|\vv\|_{L^2_x}^2 
	+ (c_1\mu +  \sum_{j=1}^3 \nu_j) \|\nabla \vv\|_{L^2_x}^2 \label{uniq-Y} 
\end{align}
where to obtain \eqref{uniq-gv-Y1} we used Young's inequality, and 
to  obtain \eqref{uniq-afterinterp} we used the bound 
\begin{equation} \label{uniq-intSob} 
	\| \vv\|_{L^4_x} \leq \|\vv\|_{L^2_x}^{\frac{1}{2}} \|\vec\nabla \vv\|_{L^2_x}^{\frac{1}{2}}, 
\end{equation} 
which follows from an interpolation followed by Sobolev embedding. 
On the other hand, we obtained \eqref{uniq-Y} via applying Young's inequality three times. 

Now we choose $\mu$ and $\nu_j$'s such that $c_1 \mu + \sum_{j=1}^3 \nu_j = 1$. Then \eqref{uniq-Y} implies that: 
\begin{equation} 
	\frac{d}{dt} \|\vv(t)\|_{L^2_x}^2
	\leq  \| \vv(t) \|_{L^2_x}^2 \left( \sum_{j=1}^2 \frac{1}{\nu_j} \|\nabla \ww_j\|_{L^2_x}^2  + \frac{c_1^2}{\mu^2 \nu_3} \|\vec g\|_{L^4_x}^4 \right), \label{uniq-beforeGron} 
\end{equation} 
which after we apply Gronwall's lemma on $[0, \rho] \subset [0, T]$ gives: 
\begin{equation} 
	\|\vv(t)\|_{L^2_x}^2
	\lesssim  \| \vv(0) \|_{L^2_x}^2 e^{ \int_0^{\rho} \left( \sum_{j=1}^2 \frac{1}{\nu_j} \|\nabla \ww_j\|_{L^2_x}^2 + \frac{c_1^2}{\mu^2 \nu_3} \|\vec g\|_{L^4_x}^4 \right) \, dt}. \label{uniq-afterGron} 
\end{equation} 
Since by the assumption each $\ww_j \in L^2([0,T], H^1) \cap L^{\infty}((0, T), L^2)$, we have 
\begin{equation} \label{uniq-G1} 
	\int_0^{\rho} \|\nabla \ww_j\|_{L^2_x}^2  \, dt < \infty, 
\end{equation} 	
for every $\rho \leq T$. 
On the other hand, by employing the assumptions  \eqref{2-assum-g1} - \eqref{lambda-L4L6},
we obtain: 
\begin{equation} \label{uniq-after-g-assum}
\int_0^{\rho} \| \vec g\|_{L^4_x}^4 \, dt \lesssim \rho^{-4\gamma} \lambda^4.
\end{equation}


Now we recall that $\vv(0) =\ww_1(0) - \ww_2(0) = 0$, and substitute that into \eqref{uniq-afterGron}, keeping in mind that estimates \eqref{uniq-G1} and \eqref{uniq-after-g-assum}
imply finiteness of the exponent on the RHS of \eqref{uniq-afterGron}. Hence we conclude that $\vv(t) \equiv 0$, which implies $\ww_1(t) = \ww_2(t)$ for all $t \in [0,T]$.  
\end{proof}

\section{Proof of the main theorems}\label{proof-main}

\noindent We find solutions $\uu$ to \eqref{NS} by writing  
$$\uu\, = \, \uu^\omega_{\vec f}\,+\, \ww $$
where we recall that $\uu^\omega_{\vec f}$ is the solution to the linear problem with initial datum 
$\vec f^\omega$ and $\ww$ is a solution to \eqref{dns0} with $\vec g=\uu^\omega_{\vec f}$. Note that $\uu$ is a weak solution for \eqref{NS} in the sense of Definition \ref{weak-sol} if and 
only if $\ww$ is a weak solution for \eqref{dns0} in the sense of Definition \ref{diff-weak-sol}. We also remark that uniqueness of weak solutions to \eqref{dns0} is equivalent to 
uniqueness of weak solutions  to \eqref{NS}. 
From now on we work exclusively with $\ww$ and the  initial value problem \eqref{dns0}.
The proof of the existence of weak solutions  is the same for both $d=2$ and $d=3$ and it is a consequence of Theorem \ref{construction-weak}. For the uniqueness claimed 
in $d=2$ we invoke Theorem \ref{lemma-uniq}.  Now to the details.

Let $\gamma<0$ be such that 
\begin{equation*}
0 < \alpha < 
\left\{\begin{array}{ll}
\frac{1}{2} + 2\gamma, \quad \mbox{ if } d = 2 \\
\frac{1}{4} + 2\gamma, \quad \mbox{ if } d = 3.
\end{array}\right.
\end{equation*}
Given $\lambda >0$, define the set 
\begin{equation}\label{setE-new}
E_{\lambda}:=E_{\lambda,\alpha,\vec f,\gamma,T} \,=\, \{\omega\in \Omega \, \, / \, \, 
\|t^\gamma\uu^\omega_{\vec f} \|_{L^4{([0,T];L^4_x)}}  \, > \, \lambda \},
\end{equation} when $d=2$.   For $d=3$, set 
\begin{equation*} 
\|  \uu^\omega_{\vec f} \|_{(\alpha,\vec f,\gamma,T)}:= \| t^{\gamma} (-\Delta)^{\frac{1}{4}} \uu^\omega_{\vec f} \|_{L^2([0,T];L^6_x)}  \, +\,  \|t^\gamma (-\Delta)^{\frac{1}{4}} \uu^\omega_{\vec f} \|_{L^\frac{8}{3}([0,T];L^\frac{8}{3}_x)}+ \|t^\gamma\uu^\omega_{\vec f} \|_{L^8([0,T];L^8_x) } 
\end{equation*} and define 
\begin{equation}\label{setE-new}
E_{\lambda}:= E_{\lambda,\alpha,\vec f,\gamma,T} \, =\, \{\omega\in \Omega \, \, / \, \,   \, \|  \uu^\omega_{\vec f} \|_{(\alpha,\vec f,\gamma,T)} \, > \, \lambda \}.
\end{equation} 
Then if we apply Proposition \ref{prop-heatfl} we find that in either case ($d=2$ or $d=3$) there exist $C_1, C_2>0$ such that
\begin{equation}\label{exp-new}
P(E_\lambda)\leq C_1\exp\left[ -C_2\left(\frac{\lambda}{C_T\|\vec f\|_{H^{-\alpha}}}\right)^2\right].
\end{equation} 

Now, let $\lambda_j=2^j, \, j\geq 0$ and define $E_j=E_{\lambda_j}$. Note that $E_{j+1}\subset E_{j}$. Then if we let
$\Sigma:=\cup E_j^c \subset \Omega$ we have that 
$$  1\geq P(\Sigma)
	=1-\lim_{j\rightarrow \infty}P(E_j)
	\geq 1-\lim_{j\rightarrow \infty}\exp\left[ -C_2\left(\frac{2^j}{C_T\|\vec f\|_{H^{-\alpha}}}\right)^2\right]=1.$$
Our goal is now to show that for a fixed divergence free vector field $\vec f\in (H^{-\alpha}(\T^d))^d$ and for  any  $\omega\in \Sigma$, 
if we define $\vec g=\uu^\omega_{\vec f}$, the  initial value problem \eqref{dns0} has a global weak solution. In fact given $\omega\in \Sigma$,
there exists $j$ such that $\omega\in E_j^c$.  In particular we then have 
\begin{equation}\label{lambdaj-L4}\|t^\gamma \vec g\|_{L^4{([0,T];L^4_x)}} \leq \lambda_j 
\end{equation} when $d=2$ and 
\begin{equation}\label{lambdaj-L8}   \| t^{\gamma} (-\Delta)^{\frac{1}{4}} \vec g \|_{L^2([0,T];L^6_x)}  + \|t^\gamma (-\Delta)^{\frac{1}{4}} \vec g\|_{L^{\frac{8}{3}}([0,T];L^{\frac{8}{3}}_x)} +  \|t^\gamma \vec g\|_{L^8{([0,T];L^8_x)}} \leq \lambda_j
\end{equation} when $d=3$.

Lemma \ref{lemma-deterministic} together with \eqref{lambdaj-L4} and \eqref{lambdaj-L8} imply that the assumptions on $\vec g$ in 
Theorem \ref{th-dns}, Theorem \ref{construction-weak} and Theorem \ref{lemma-uniq} are satisfied. This concludes the proof.



\end{document}